\documentclass[12pt,a4paper]{amsart}
\allowdisplaybreaks 
\usepackage{hyperref}
\usepackage{amsmath}
\usepackage{amssymb}
\usepackage{amsthm}
\usepackage{amsfonts}
\usepackage{latexsym}

\newcommand\Lnorm[2]{\left\Vert{#1}\right\Vert_{L^{#2}}}
\def\00{{\infty}}
\newcommand\set[1]{\left\{#1\right\}}
\newcommand\pref[1]{~(\ref{#1})}

\theoremstyle{plain}
\newtheorem{theorem}{Theorem}[section]
\newtheorem{corollary}[theorem]{Corollary}
\newtheorem{lemma}[theorem]{Lemma}
\newtheorem{proposition}[theorem]{Proposition}
\theoremstyle{definition}
\newtheorem{definition}{Definition}
\theoremstyle{remark}

\numberwithin{equation}{section}

\newcommand\bfc{\mathbf{c}}

\def\cX{\mathcal X}

\def\cM{\mathcal M}

\def\cD{\mathcal D}

\def\C{\mathbb C}
\def\R{\mathbb R}

\def\la{\lambda}
\def\a{\alpha}
\def\b{\beta}
\def\vp{\varphi}

\def\dmu{d\mu}

\def \RE {\text{\rm Re}\,}
\def \IM {\text{\rm Im}\,}

\def\inutile#1{{}}

\begin{document}

\title[Dynamics of the heat semigroup in Jacobi analysis]
{Dynamics of the heat semigroup in Jacobi analysis}

\author[F. Astengo, B. Di Blasio]
{Francesca Astengo and Bianca Di Blasio}

\address
{Dipartimento di Matematica\\
Via Dodecaneso 35\\
16146 Genova\\ Italy} \email{astengo@dima.unige.it}

\address
{Dipartimento di Matematica e Applicazioni\\
Via Cozzi 53\\
  20125 Milano\\ Italy}
\email{bianca.diblasio@unimib.it}

\thanks{Work partially supported
by  MIUR (project ``Analisi armonica").
}





\begin{abstract}
Let $\Delta$ be the Jacobi Laplacian. 
We study the chaotic and hypercyclic behaviour
of the strongly continuous semigroups of operators generated
by perturbations of $\Delta$ with a multiple of the identity on $L^p$ spaces.
 \end{abstract}

\maketitle

\section{Introduction}

%
 
Chaos in the context 
of strongly continuous semigroups of bounded linear operators in Banach spaces
has been introduced by W.~Desch, W. Schappacher and G.~Webb~\cite{DSW}
 as a generalization to
 continuous time of the discrete time case.
 In their paper~\cite{DSW}, the authors give a sufficient condition for a strongly continuous semigroup to be chaotic in terms of 
the spectral properties
of its infinitesimal generator.
Moreover they study in detail many examples, mainly the
transport equation and second order differential operators
with constant coefficients (see also~\cite{dLEG}).

The purpose of this paper is to apply the criterion in~\cite{DSW}
to study of the dynamics of the (modified)
heat semigroup  generated by the Jacobi operator
(i.e., the generator is a
perturbation of the Jacobi operator by a multiple of the identity),
which is a second order differential
 operator with nonconstant
coefficients.

Jacobi analysis can be developed as a generalization
of the Fourier-cosine transform and has been studied
by many authors (see~\cite{Koo}), the main interest being the interplay
between the analytic and geometric properties of the Jacobi operator.
Indeed, in certain cases, the Jacobi operator 
is the radial part of the Laplace--Beltrami operator
on Damek--Ricci spaces~\cite{ADY},
therefore Jacobi analysis includes radial analysis
on symmetric spaces of real rank one as a special case.

The dynamics of the (modified) heat semigroup
on non compact
  symmetric spaces was already studied
in~\cite{JiWeber1, JiWeber2}.
We extends the results in~\cite{JiWeber1} to the context  of 
Jacobi analysis. Therefore,  as particular cases, we cover Damek--Ricci spaces and Heckman--Opdam root spaces of rank one,
which were not treated in~\cite{JiWeber1}.

The paper is organized as follows: in Section~2 we 
settle notation and recall  some  basic facts regarding
chaotic semigroups, Jacobi analyis and Lorentz spaces.
In Section~3 we establish some properties of spherical
functions, the weak type (1,1) boundedness of the heat maximal function
and the $L^{p,q}$ inversion formula.
In Section~4 we apply the results in the previous section
to the study of the dynamics of
(modified) heat semigroups.

\section{Notation and Preliminaries}

\subsection{Chaotic semigroups}

In this paper we follow R.~L.~Devaney~\cite{Dev}, who has defined
 chaos in metric spaces in the following sense. 
  A continuous map $f$ on a metric space $X$ is said to be 
chaotic if it is topologically transitive, i.e., some element has
a dense orbit,
and if the set of its periodic points is dense in $X$.
These two conditions imply (see~\cite{Monthly})
that $f$ has sensitive 
dependence 
on initial 
conditions.

  In~\cite{DSW} the authors have generalized
this definition to strongly
continuous semigroups as follows.

\begin{definition}
 A strongly continuous semigroup $\{T(t) : t\geq 0\}$ on a Banach space $\cX$ 
is said to be {\em hypercyclic} if there exists $f$ in $\cX$ such that its orbit  
$\{T(t)f : t\geq 0 \}$ is dense in~$\cX$.
\end{definition}

\begin{definition}
 A strongly continuous semigroup $\{T(t) : t\geq 0\}$ on a Banach space $\cX$ is said to be  {\em chaotic} if it is hypercyclic and   the set of periodic points
$$
\{ f\in \cX : \exists t>0 \mbox{~such that~} T(t)f=f \}
$$	 is dense in $\cX$.		   
 \end{definition}

 We denote by $\sigma(B)$ and   $\sigma_{pt}(B)$ respectively the spectrum and the   point spectrum of a linear operator $B$ on a Banach space $\cX$.
 
A sufficient condition for a strongly continuous semigroup to be chaotic in terms of 
the spectral properties
of its generator was given by Desch, Schappacher, and Webb~\cite[Theorem~3.1]{DSW}: 
\begin{theorem}
\label{thm dsw}
Let $B$ be the infinitesimal generator of
a strongly continuous semigroup $\{T(t) : t\geq 0\}$  on a separable 
 Banach space $\cX$. 
 Let $\Omega$ be an open connected subset of $\sigma_{pt}(B)$ which intersects the imaginary axis. For each $z$ in $\Omega$ let $\phi_z$ be a nonzero eigenvector, i.e. 
 $B \phi_z=z\, \phi_z$. 
   Suppose that for every $f$ in the dual space $\cX'$ of $\cX$  the function $F_{f}:\Omega\to \C, $ defined by 
 $$
 F_{f}=\langle f,\phi_z \rangle
 $$
  is analytic  and   does not vanish identically unless $f=0$. Then the semigroup $\{T(t) : t\geq 0\}$ is chaotic.
  \end{theorem}
  
  Many authors studied
 sufficient conditions for a strongly 
	continuous semigroups to be hypercyclic
	(see~\cite{JiWeber1} and the references therein).

\subsection{Lorentz spaces}
Let $f$ be a measurable function on the measure space $(X,\cM,\mu)$. The \emph{nonincreasing
rearrangement} of $f$ is the function $f^*$ on $\R^+$ defined by
\[
f^*(t)=\inf\left\{s\in\R^+\,\,: \,\,
   \mu\left(\{x\in X\, :\,
|f(x)|>s\} \right) \, \leq t \right\} \qquad \forall t\in \R^+.
\]
The function $f^*$ is nonincreasing, nonnegative, equimeasurable
with $f$ and right-continuous. For any given measurable function
$f$ on $X$, we define
$$
\Lnorm{f}{p,q}=\left(\frac{q}{p}\int_0^\00 \left(s^{1/p}\,
f^*(s)\right)^q\, \frac{ds}{s} \right)^{1/q} \qquad 1\leq p<\00,\quad \qquad
1\leq q<\00,
$$
and
$$
\Lnorm{f}{p,\00}=\sup\left\{s^{1/p}\,
f^*(s) \, : \, s\in \R^+ \right\} \qquad 1\leq p< \00.
$$

\begin{definition} Let $1\leq p<\00$ and $1\leq q\leq\00$.
The Lorentz space $L^{p,q}(\mu)$
 consists of those measurable functions
$f$ on $X$ such that $\Lnorm{f}{p,q}$ is finite.
\end{definition}

It is easy to check that  $L^{p,p}(\mu)$ coincides
with the usual Lebesgue
space~$L^p(\mu)$, with equality of norms. Moreover, if $q_1<q_2$, then
$L^{p,q_1}(\mu)$ is contained in $L^{p,q_2}(\mu)$ and, if $1<p,q<\00$, the
dual space of $L^{p,q}(\mu)$ is $L^{p',q'}(\mu)$.
Here and elsewhere in this paper $p'$ and $q'$ 
are the conjugate exponents of $p$ and $q$, i.e.
$\frac1p+\frac1{p'}=\frac1q+\frac1{q'}=1$.

A good reference for Lorentz spaces is \cite{H}. We recall the
following
multiplication theorem from that paper. 
\begin{lemma} \textnormal{\cite[p.~271]{H}} \label{Huntmultipl}
Let $p_0,p_1$ and $q_0,q_1$ be
in $[1,\00]$. Then there exists a constant $C$
 such that for every $f$ in
$L^{p_0,q_0}(\mu)$ and $m$ in $L^{p_1,q_1}(\mu)$
$$
\Lnorm{mf}{p,q}\leq C\, \Lnorm{m}{p_1,q_1}\, \Lnorm{f}{p_0,q_0},
$$
where $\frac{1}{p}=\frac{1}{p_0}+\frac{1}{p_1}$ and
$\frac{1}{q}=\frac{1}{q_0}+\frac{1}{q_1}$.
\end{lemma}

	\subsection{Jacobi Analysis}\label{Jacobi}
We recall some facts about Jacobi analysis 
which we shall need in the sequel. 
We follow Koornwinder~\cite{Koo} and the normalizations therein.
Throughout this paper, $\a,\b$ will be real numbers, with $\a\geq\b\geq-\frac12$, 
$\a>-\frac12$. We define
\[
A(x)=(2\sinh x)^{2\a+1}\,(2\cosh x)^{2\b+1}
\qquad\forall x>0.
\]

For complex $\lambda$ and $\rho=\a+\b+1$, let 
$$
\Delta =- \left(\frac{d}{dx}\right)^2-
\frac{A'(x)}{A(x)}\frac{d}{dx} 
$$
and  consider the differential equation
\begin{equation}
\label{eq:Jacobi}
\Delta u (x)=(\la^2+\rho^2)\, u(x)
\qquad  x>0.
\end{equation}
Using the substitution $z=-\sinh^2x$, we can transform
equation\pref{eq:Jacobi} in the well known hypergeometric differential
equation
\[
z(1-z)\,u''(z)+(c-(a+b+1)z)\,u'(z)-ab\,u(z)=0
\]
of parameters $a=\frac{1}{2}(\rho-i\lambda)$, $b=\frac{1}{2}(\rho+i\lambda)$,
$c=\a+1$.

Let $ {_2}F_1$ denote the Gaussian hypergeometric function.
The Jacobi function~$\varphi_\lambda =\varphi_\lambda^{(\alpha,\beta)}$
 of order $(\alpha,\beta)$    
 $$
 \varphi_\lambda (x)=
 {_2}F_1(\tfrac{1}{2}(\rho-i\lambda),
\tfrac{1}{2}(\rho+i\lambda); \alpha+1,-\sinh^2x)
\qquad x \in \R
 $$
 is the unique even smooth function on $\R$     which satisfies
 \hbox{$u(0)=1$} and the differential equation\pref{eq:Jacobi}.  
 
Therefore the function $\lambda\mapsto\varphi_\lambda(x)$ is analytic for all
$x\in \R  $. 
Moreover $\{\varphi_\lambda \}$ is a continuous orthogonal system on $\R^+$ with respect to the weight function $A$.

We consider on $\R^+=(0,\infty)$ the measure $\mu$ which is absolutely continuous
with respect to the ordinary Lebesgue measure and has density $A$. 
When $1\leq p\leq\infty$ we denote by $L^p(\mu)$ the ordinary Lebesgue space on $\R^+$
with respect to the measure $\mu$.

Note that
\begin{equation}
\label{crescitaA}
|A(x)|\leq
C\, \begin{cases}
x^{2\a+1} &0<x<1\\
e^{2\rho x} &x\geq 1.
\end{cases}
\end{equation}

The Jacobi transform $f\longmapsto \hat f $ is defined by
\[
\hat f(\la)=
\int_{\R^+}f(x)\, \vp_\lambda (x)\,\dmu(x),
\]
for all  functions~$f$ on $\R^+$ and complex numbers $\lambda$
for which the right hand side is well defined.

Let  $\cD^\sharp(\R)$ be the space of   smooth even functions on $\R$ with 
compact support and denote by $\cD^\sharp(\R^+)$ 
the space of the restrictions to $\R^+$ of functions in 
$\cD^\sharp(\R)$. 
Note that 
$$
(\Delta f)\hat{\phantom f}=(\la^2+\rho^2)\hat f,
\qquad \forall f\in 
\cD^\sharp(\R^+).
$$

The following inversion formula holds
for functions in $\cD^\sharp(\R)$~\cite[p. 9]{Koo}
\[
f(x)=\frac1{2\pi} \,\int_0^\infty\hat f(\la)\, \vp_\lambda (x)\,|\bfc(\la)|^{-2} d\la,
\qquad \forall x\in \R
\]
where $\bfc(\la)$ is a multiple of the  meromorphic Harish-Chandra function
given by the formula
\[
\bfc(\la)
=
\frac{2^{\rho-i\lambda}\Gamma(\a+1)\Gamma( i\lambda)}
{\Gamma(\frac{1}{2}(\rho+i\lambda))\Gamma(\frac{1}{2}(\rho+i\lambda)-\beta)}.
\]
Moreover the Jacobi transform $f\longmapsto \hat f $
extends to an isometry from $L^2(\mu)$ onto $L^2(\R, \frac1{2\pi} |\bfc(\la)|^{-2} d\la)$.

The following formula holds
\begin{equation}\label{prodotto}
\vp_\la(x)\, \vp_\la(y)=\int_0^\infty \vp_\la(u)\, W(x,y,u)\, \dmu(u)
\end{equation}
where the kernel $W$ is explicitly known (see~\cite[p.~58]{Koo}).
Thus 
one can define the (generalized) translation operator
\[
\tau_xf(y)
=\int_0^\infty f(u)\, W(x,y,u)\, \dmu(u)
\qquad\forall f\in
\cD^\sharp(\R^+)
\]
and the (generalized) convolution
of functions $f,g$ in $\cD^\sharp(\R^+)$, say, by
\[
f\star g(x)=\int_0^\infty \tau_xf(y)\, g(y)\, \dmu(y),
\]
so that 
$$
{(f\star g)}\hat{\phantom f}=\hat f\,\,  \hat g.
$$
The function $W$ is nonnegative, supported  in
$|x-y|\leq u\leq x+y$ and symmetric in its three  variables,
so that $f\star g=g\star f$. 
From\pref{prodotto} with $\la=i\rho$ it follows 
\[
\int_0^\infty
W(x,y,u)\,  \dmu(u)=1.
\]

Moreover $\int_0^\infty\tau_xf\,\dmu=\int_0^\infty f\,\dmu  $ and $\tau_0f=f$.  Finally, 
   the Young inequality holds~\cite[p. 61]{Koo}
 \begin{equation}
\label{young}
\|f\star g\|_r\leq C\, \|f\|_p\,\|g\|_q
\qquad \frac1{p}+\frac{1}{q}=1+\frac{1}{r}.
\end{equation}

\subsection{Heat kernel}

Let $f$ be in $\cD^\sharp(\R^+)$ 
and
consider the initial value problem for the heat equation
\begin{equation}
\label{IVPcalore}
\begin{cases}
\Delta u(t,x)=\partial_t u(t,x) 
\\
\lim_{t\to 0^+}u(t,x)=f(x).
\end{cases}
\quad t>0, \, x\in \R
\end{equation}
One can show that the  solution $u$ to the problem\pref{IVPcalore}
can be written as 
$u(t,x)=f\star h_t(x)$, where the function 
$h_t$ is called  {\em heat kernel} and it is defined    on the Jacobi transform side by 
\[
\widehat{h_t} (\la) = 
e^{-t(\la^2+\rho^2)} \qquad \forall t>0,\quad \la\in \C.
\]

The heat kernel  $h_t$ is a nonnegative  function   in $L^p(\mu)$ for all $p\geq 1$ and
$\|h_t\|_1=1$~\cite{Trimeche}.

 For every $p$ in $[1,\infty]$ denote by $\Delta_p$ the closure in $L^p(\mu)$ of the operator $\Delta$
with domain $\cD^\sharp (\R^+)$.
By the Young inequality\pref{young}, the family of operators $\{H_p(t)\}_{t>0}$ defined by
$$
H_p(t)f=h_t\star f\qquad \forall f\in L^p(\mu)
$$
is a strongly continuous symmetric semigroup on $L^p(\mu)$
whose generator is $\Delta_p$.

The sharp estimate for the heat kernel was established in~\cite{Kawazoe}

\begin{equation}\label{stimaht}
h_t(x)\simeq  \,t^{-\a-1}\,(1+t+x)^{\a-\frac12}\,
(1+x)\, e^{-\rho x-\rho^2 t}\, e^{-x^2/4t}
\qquad
\forall x>0,\, t>0.
\end{equation}
Here $f(x)\simeq g(x)$ stands for $C_1g(x)\leq f(x)\leq C_2g(x)$, for some constants $C_1$ and $C_2$.

\section{The   $L^p$ inversion formula}

Since for real values of $\lambda$, we have $|\vp_\la|\leq \vp_0\leq 1$ (see~\cite[p.~53]{Koo}),
the function space 
\[
L^1_0 (\mu) =\set{f:\R^+\to\C\, :\,
f\ \text{measurable\ and\ }\,
f\vp_0\in L^1(\mu)}
\] 
contains $L^1(\mu)$ and the Jacobi transform $  \hat f$ is well defined as a function on $\R$ when $f$ is in $L^1_0 (\mu)$.

In the next results we show that the Lorentz spaces 
$L^{p,q} (\mu )$, $1 < p < 2$, $1 
\leq q \leq \infty$ are subspaces of $L^1_0(\mu)$.

When $1\leq p<\infty$ we define
$$
S_p = 
\set{\la\in \C\, :\, |\IM(\la)|\leq \left| 1-\tfrac2{p}\right|\,\rho
} .
$$
Note that $S_p=S_{p'}$, when $p>1$. By $S^\circ_p$ and $\partial S_p$
 we denote respectively the interior and the boundary of $S_p$. 
It is well known that from the estimate~\cite[p. 53]{Koo}
\begin{equation}
\label{stimasferiche}
|
\varphi_\lambda(x)|\leq C\, (1+x)\, e^{(|\IM(\la)|-\rho)x}
\qquad \forall x\geq 0, \la\in \C,
\end{equation}
it follows that if $1<p< 2$ and  $\la$ is in   $S^\circ_p$
then $\vp_\la$ is in $L^{p'}(\mu)$. A more precise
result holds (see~\cite{RS} for the case of Damek--Ricci spaces).
\begin{lemma}\label{sferichelpq}
The following estimates hold.
 \begin{itemize}
\item[(i)] Let $1 < p < 2$. If $\la$ is in   $S^\circ_p$
then $\vp_\la$ is in $L^{p',q}(\mu)$ for any $q$ 
in $[1, \infty]$. 
\item[(ii)] If $1\leq p<2$ and $\la$ is in $\partial S_p$
then $\vp_\la$ is in $L^{p',\infty}(\mu)$.
\item[(iii)] If $p =2$ and $\la$ is real, then $\vp_\la/(1+x)$ is in $L^{2,\infty}(\mu)$.\end{itemize}
\end{lemma}

\begin{proof} Let $1<p<2$.
It suffices to show that (i) holds with $q=1$. By~\pref{stimasferiche},
when $\la$ is in   $S^\circ_p$ for some $\varepsilon>0$,
\[
|\varphi_\lambda (x)|\leq C\,
e^{-(\frac2{p'}+2\varepsilon)\rho x} 
\qquad\forall x \in \R^+,
\]
therefore
\[
\vp_\la^*(t)\leq C\,\psi^*(t) \qquad \forall t\in \R^+,
\]
where $\psi(x)=e^{-(\frac2{p'}+2\varepsilon)\rho x} $.
We now compute $\psi^*$. By equation\pref{crescitaA} we have 
\begin{align*}
\psi^*(t)&=\inf
\set{s\in \R^+\,:\,\int_0^{\log(s^{-1/(2/p'+2\varepsilon)\rho)})}A(x)\, dx\leq t}
\\
&\leq C\,
\begin{cases}
 t^{-(\frac{1}{p'}+\varepsilon)} &t>1
 \\
 \text{const} &t<1.
\end{cases}
\end{align*}
We conclude that $\psi^*$ belongs to $L^{p',q}(\mu)$ for any $q$ 
in $[1, \infty]$.

The proofs of (ii) and (iii) are similar. 
If $p =2$ and $\la$ is real we apply again~\pref{stimasferiche}.

If $|\IM(\la)|=(2/p-1 )\rho$ and $1\leq p<2$, we use the  inequality~\cite{FJ} 
\begin{equation}
|
\varphi_\lambda (x)|\leq C_\delta\,   e^{-(|\IM(\la)|+\rho)x}
\qquad \forall x\geq \delta>0.
\end{equation}
  
\end{proof}

Denote by $P_p$, $1\leq p<\infty$, the parabolic region 
\[
P_p=\set{\la^2+\rho^2\, :\, \la\in S_p^\circ}.
\]
 
\begin{corollary}\label{parabolica}
Let $2<p<\infty$. For any $z$ in the   parabolic region $
P_p$ there exists a nonzero $\phi_z$ in $L^{p}$ such that $\Delta\phi_z=z\phi_z$.
\end{corollary}

Note that Corollary~\ref{parabolica} implies that the  parabolic region $P_p$ is contained
the point 
spectrum $\sigma_{pt}(\Delta_{p})$, for $2<p<\infty$.

\begin{corollary}
$L^{p,q}(\mu)\subset L^1_0(\mu)$, $1<p<2$, $1\leq q\leq \infty$.
\end{corollary}

\begin{proof}
Use the multiplication theorem with $f$ in $L^{p,q}$
and $m=\vp_0$ which is in $L^{p',q'}$.
\end{proof}

So, when $f$ is 
in $L^{p,q}$,  we can define the Jacobi transform $\hat f$ as a function 
on $\R$ and actually $\hat f$ turns to be  holomorphic in 
$S^\circ_p$.

\begin{corollary}\label{olotrasf}
If $f$ is in $L^{p,q}(\mu)$, $1<p<2$, $1\leq q\leq \infty$,
then $\widehat f$ is a bounded function
in the strip $S_p$ holomorphic in  $S^\circ_p$.
\end{corollary}

\begin{proof}
Use   Lemma~\ref{sferichelpq} and the Morera Theorem.
\end{proof}
We include the next result, which we were unable to find in the literature (see~\cite{ADY} in the case of Damek--Ricci spaces).
\begin{proposition}
The heat maximal operator 
$$
H^*f(x)=\sup_{t>0}|h_t\star f(x)|
\qquad f\in \cD^\sharp(\R^+)
$$
is of weak type $(1,1)$ and of strong type $(p,p)$.
\end{proposition}

\begin{proof}
Since $\|h_t\|_1=1$ for every $t>0$ and by\pref{young},
the heat maximal operator is trivially bounded on $L^\infty(\mu)$.
We now prove the weak $L^1$-estimate. From this estimate
and the Marcinkiewicz Interpolation Theorem the thesis follows.

As usual, we shall deal with the small time
maximal function
$$
H^*_0f(x)=\sup_{0<t\leq 1}|h_t\star f(x)|
\qquad f\in \cD^\sharp(\R^+)
$$
and the large time maximal function
$$
H^*_\infty f(x)=\sup_{t> 1}|h_t\star f(x)|
\qquad f\in \cD^\sharp(\R^+)
$$
separately. 
Note that 
\[
H^*_\infty f\leq |f|\star [\sup_{t> 1}h_t].
\]
From the heat kernel estimate\pref{stimaht} it follows that
$\sup_{t>1}h_t(x)=O((1+x)^{-1/2}\, e^{-2\rho x})$
when $x$ is large
so that $H^*_\infty$ is of weak type $(1,1)$. Indeed,
if  $k(x)=\cosh(x)^{-2\rho}$,  
Liu~\cite[Lemma~3.2]{Liu} proved that
\begin{align*}
\tau_x k(y)
\leq C\, \cosh(x)^{-2\rho} \qquad\forall y>0,
\end{align*}
so that
\begin{align*}
|f|\star k(x)
=\int_0^\infty  |f(y)|\, \tau_x k(y)\, A(y)\, dy
\leq 
\cosh(x)^{-2\rho}\|f\|_1.
\end{align*}
Therefore
\begin{align*}
\mu\set{x\, :\, |f|\star k(x)>\lambda }
&\leq 
\mu\set{x\, :\,(\cosh x)^{2\rho}< C\,\|f\|_1/\lambda }
\\
&=\int_0^{x_0}A(u)\, du
\\
&\leq C\, (\cosh x_0)^{2\rho}= C\,\|f\|_1/\lambda.
\end{align*}
where $x_0>0$ is such that $(\cosh x_0)^{2\rho}= C\,\|f\|_1/\lambda$ (if any,
otherwise $x_0=0$ and the weak type inequality is trivial).

In order to prove that $H^*_0$ is of weak type $(1,1)$,
we first note that the estimates of the heat kernel imply that
when $0<t\leq 1$,
\begin{align*}
&0\leq h_t(x)\leq C\, e^{-x^2/4}=k_1(x)
\qquad &x>1
\\
& h_{t}(x)\simeq  \,k_t(x)
\qquad &0<x\leq 1,
\end{align*}
where 
$$
k_t(x)=t^{-(\a+1)}\,e^{-x^2/4t}
\qquad t,x>0.
$$
Let $\chi$ denote the characteristic function of the interval
$[-1,1]$,
 and write
\[
H^*_0f(x)
\leq \sup_{0<t\leq 1}|((1-\chi)h_t)\star f(x)|+ \sup_{0<t\leq 1}|(\chi h_t)\star f(x)|
\leq k_1\star |f|(x)+ C\,\sup_{0<t\leq 1} (\chi k_t)\star |f|(x)
\]
Since the kernel $k_1$ is integrable,
the operator $f\mapsto k_1\star f$ is $L^1$-bounded.

For the estimate of the other term we will use the
weak type $(1,1)$ boundedness of a Hardy--Littlewood maximal function. Let $X_r$ denotes the characteristic function of $[-r,r]$,
normalized so that $\int X_r\, d\mu=1$ and define
 $$
 Mf=\sup_{r>0}X_r\star |f|.
 $$
Liu~\cite{Liu} proved that the operator $M$ is of weak type $(1,1)$.
 
 Let $\nu(y)=\int_0^y   A(x)\, dx$ and 
   $G(y)= \int_0^y \tau_x|f|(u)\,  A(u)\, du$. 
   Note that  $G(y)\leq \nu(y)\, M\tau_x|f|(0)$. Applying 
   the size estimate\pref{crescitaA}, we get

\begin{align*}
 (\chi k_t)\star |f|(x)
 &=
\int_0^1 \tau_x|f|(y)\, k_t(y)\,   d\mu(y)
\\
&=-\int_0^1 G(y)\, k'_t(y)\, dy+G(1)k_t(1)
\\
&\leq M\tau_x|f|(0)\int_0^1\nu(y)  (- k'_t(y))\,   dy +
G(1)k_t(1)
\\
&\leq M\tau_x|f|(0)\int_0^1\nu'(y)   k_t(y)\,   dy
\\
&= Mf(x)\int_0^1    k_t(y)\,   A(y) dy
\leq  
 Mf(x).
\end{align*}
 The thesis follows.
\end{proof}

The standard method of approximation using the heat kernel gives the following 
inversion formula for functions in the Lorentz space $L^{p,q}(\mu)$, 
with $1<p<2$, $q\geq 1$ (see~\cite{RS} for the case of Damek--Ricci spaces).

\begin{proposition}
Let $f$ be in $L^{p,q}(\mu)$, with $1<p<2$, $q\geq 1$ 
or $f$ in $L^1\cup L^2(\mu)$. 
If $\hat f$ is in $L^1(\R,|\bfc(\la)|^{-2}\,d\la)$,
then for almost every $x$ in $\R^+$,
\begin{equation}\label{inversioneLpq}
f (x) = \frac1{2\pi} 
\int_0^\infty \hat{f}(\la)\,\vp_\la(x)\, |\bfc(\la)|^{-2}\,d\la.
\end{equation}
\end{proposition} 
\begin{proof}
We can write  $f = 
f_1 + f_2$, where $f_1$ is in $L^1(\mu)$
and $f_2$ is in $L^2(\mu)$
~\cite{H}.
Then for every $t>0$, we have $f\star h_t= f_1\star h_t+f_2\star h_t$.
Since the heat maximal operator is of weak type
(1, 1)
and $g\star  
h_t (x) \to g(x)$, for $t\to 0$, whenever $g$ is smooth and compactly supported,
 it follows that there exist measurable sets $E_1 , E_2$ 
of null measure such that
$ f_1 \star  h_t (x) \to f_1 (x)$ for all $x$ in $^cE_1$ 
and $f_2\star  h_t (x) \to f_2 (x)$ for all $x$ in $^cE_2$,  as $t\to 0$. 
This implies that if  $x$ is not in $E_1\cup E_2$ then
$ f \star  h_t (x) \to f (x)$ as $t\to 0$.
Moreover $\mu(E)\leq\mu(E_1)+\mu(E_2)=0$.

Since $\hat{f}$ is in $L^1(\R,|\bfc(\la)|^{-2}\,d\la)$,
for every test function $\psi$ we have 
\begin{align*}
\langle f\star h_t,\psi\rangle
=\frac1{2\pi} \int_0^\infty
\int_0^\infty \hat{f}(\la)\,e^{-t(\la^2+\rho^2)}\,\vp_\la(x)\, |\bfc(\la)|^{-2}\,d\la
\,\psi(x)\,  d\mu (x).
\end{align*}
Using the Dominated Convergence Theorem we now get the result.
\end{proof}

\section{Dynamics of the heat semigroup}
 
 For every $p$ in $[1,\infty]$ denote by $\Delta_p$ the closure in $L^p(\mu)$ of the operator $\Delta$
with domain $\cD^\sharp (\R^+)$.
As previously observed the family 
$$
 e^{-t\Delta_p}f=h_t\star f\qquad \forall f\in L^p(\mu)
$$
is a strongly continuous semigroup on $L^p(\mu)$.

In this section we study the dynamics of the shifted semigroup
$$
   e^{-t(\Delta_p-\theta)}: L^p(\mu) \to L^p(\mu).
$$

\begin{theorem}\label{chaotic} Let $2<p<\infty$. Then  for all $\theta> \theta_p=
  \rho^2 -  \rho^2\,\left(\frac{2}{p}-1\right)^2$ the semigroup
  $$
   e^{-t(\Delta_p-\theta)}: L^p(\mu) \to L^p(\mu)
  $$
  is chaotic.
 
  \end{theorem}

\begin{proof} We apply Theorem~\ref{thm dsw}.
 By Corollary\pref{parabolica}, the parabolic region
 $P_p$ is contained in 
 the point spectrum  $\sigma_{pt}(\Delta_p)$  and the corresponding eigenfunctions are given by 
 the appropriate Jacobi functions.  The vertex of the parabolic region $P_p$ is at the point
 $$
  \theta_p = \rho^2 -  \rho^2\,\left|\frac{2}{p}-1\right|^2
 =
  4\rho^2\, \frac1{pp'}
 $$ 
 and hence the point spectrum of 
 $(\Delta_p - \theta)$ intersects the imaginary axis for all
 $\theta > \theta_p$.

 Suppose now that $\theta>\theta_p$ and let $\Omega_\theta$ denote the set
 $
  \left(P_p -\theta \right)\setminus \{z\in \R : z\leq \rho^2 - \theta \},
 $
 i.e.
 $$
 \Omega_\theta=
 \left\{
 z\in\C\,\,:\,\, z=\la^2+\rho^2-\theta\quad |\IM(\la)|<( 1-\tfrac2{p} )\rho
  \right\}\setminus \{z\in \R : z\leq \rho^2 - \theta \}.
 $$
 Then $\Omega_\theta$ is an open, connected subset of the point spectrum of 
 $(\Delta_p - \theta)$
 that intersects the imaginary axis.
 
 Since $(\Omega_\theta +\theta -\rho^2) \cap \{x\in \R\,\,:\,\, x<0\}= \emptyset$  we choose
 an analytic branch $\sqrt{z + \theta -\rho^2}$
  of the square root    so that $\RE \sqrt{z + \theta -\rho^2}>0$, for every $z$ in  $ \Omega_\theta$.
   Note that  $z\longmapsto \sqrt{z + \theta -\rho^2}$ maps 
  $\Omega_\theta$ onto the open strip
  $$
  \left\{\la\in \C :  \RE \la>0, |\IM \la| < \rho\,  ( 1-\tfrac{2}{p})\right\}.
  $$
 For every $z$ in  
  $\Omega_\theta$
  we choose the eigenfunction $\phi_z$ defined by 
 
\begin{equation}\label{phiz}
\quad \phi_z=\vp_\la \quad \text{where } z=\la^2+\rho^2-\theta, 
\end{equation}
so that  $ (\Delta_p-\theta) \phi_z=z\phi_z$.
As in Theorem~\ref{thm dsw}, for every  
 $f$ in $  L^{p'}(\mu)$, we define
   the function $F_f\,:\,
\Omega_\theta\longrightarrow \C$  by
$$
F_f(z)=\langle f, \phi_z\rangle=\hat f(\sqrt{z + \theta -\rho^2}).
$$
Then by Corollary~\ref{olotrasf} the function  $F_f$ is holomorphic since it 
is the composition of two holomorphic functions. Moreover
  $F_f =0$ implies $f=0$ by the inversion formula\pref{inversioneLpq}.

\end{proof}

In the following theorem we  prove  that the semigroup
     $$
       e^{-t(\Delta_p-\theta)}: L^p(\mu) \to L^p(\mu)
     $$
         is not chaotic when $1< p \leq 2$ for every  $\theta$ in $\R$.
         
\begin{theorem}\label{notchaotic}
  Let $1< p \leq 2$ and $\theta$ in $\R$. Then the semigroup
     $$
       e^{-t(\Delta_p-\theta)}: L^p(\mu) \to L^p(\mu)
     $$
         does not have periodic elements. Moreover  when $1< p < 2$ it is not hypercyclic.
\end{theorem}

\begin{proof}
 Let $1<p\leq 2$, $\theta$ in $\R$ and   $f$   a periodic point in $L^p(\mu)$ for  $e^{-t(\Delta_p- \theta)}$. Then there exists $t>0$
such that $h_t\star f=f$ or equivalently 
$
\left(e^{-t(\la^2 +\rho^2-\theta)}-1\right)\,\widehat f(\lambda)=0
$
for every $\lambda$ in $S^o_p$ when $1<p<2$ and for almost every real  $\lambda$  when $p=2$.
By the inversion formula\pref{inversioneLpq}  $f=0$ and the first part follows.
 
Let $1<p<2$ and
assume that the semigroup $e^{-t(\Delta_p-\theta)}$ is hypercyclic. Then, the dual operator
$(\Delta_p - \theta)' = \Delta_{p'}-\theta$ of its generator would have empty point spectrum
~\cite[Theorem~3.3]{DSW}. This a contradiction and the  thesis follows. 
\end{proof}


\begin{thebibliography}{9999}

\newcommand\auth[1]{{\textrm{#1}}}
\newcommand\papertitle[1]{{\textrm {#1}}}
\newcommand\jourtit[1]{{\textit{\frenchspacing#1}}}

\newcommand\journum[1]{{\textbf{#1}}}
\newcommand\booktit[1]{{\textit{#1}}}



\bibitem{Trimeche}
	\auth{A.~Achour and K.~Trim\`eche},
	\papertitle{La $g$-fonction de {L}ittlewood-{P}aley associ\'e \`a un op\'erateur diff\'erentiel singulier sur $(0,\infty)$},
	\jourtit{Ann. Inst. Fourier (Grenoble)}
	\journum{33}
	 (1983), 203--226.



\bibitem{ADY}
\auth{J.--Ph.~Anker, E.~Damek, and C.~Yacoub},
\papertitle{Spherical analysis on harmonic $AN$ groups},
\jourtit{Ann. Scuola Norm. Sup. Pisa}
\journum{23}
(1996),
 643--679.


\bibitem{Monthly}
\auth{J.~Banks, J.~Brooks, G.~Cairns, G.~Davis, and P.~Stacey}, 
\papertitle{On {D}evaney's
  definition of chaos}, 
  \jourtit{Amer. Math. Monthly}
   \journum{99} (1992), 
  332--334. 

	     

\bibitem{dLEG}
\auth{R.~deLaubenfels, H.~Emamirad, and K.-G.~Grosse-Erdmann}, 
\papertitle{Chaos for
  semigroups of unbounded operators}, 
  \jourtit{Math. Nachr.}
  \journum{261/262} (2003),
  47--59. 


\bibitem{DSW}
	\auth{W.~Desch, W.~Schappacher, and G.~F.~Webb},
	\papertitle{Hypercyclic and
  chaotic semigroups of linear operators},
	\jourtit{Ergodic Theory Dynam. Systems}
    \journum{17}
     (1997),  793--819.

\bibitem{Dev}
	\auth{R.~L.~Devaney},
	\booktit{An {I}ntroduction to {C}haotic {D}ynamical {S}ystems},
	second
  ed., Addison-Wesley Studies in Nonlinearity, Addison-Wesley Publishing
  Company Advanced Book Program, Redwood City, CA, 1989.
 
 

\bibitem{FJ}
\auth{M.~Flensted-Jensen},
\papertitle{{P}aley--{W}iener type theorems for a differential operator connected
  with symmetric spaces},
\jourtit{Ark. Mat.} \journum{10} (1972), 143--162.


 \bibitem{H}
   \auth{R. A. Hunt},
   \papertitle{On $L(p,q)$ spaces},
   \jourtit{L'Enseignement Math.}  \journum{12}
   (1966), 249--276.


\bibitem{JiWeber1}
	\auth{L.~Ji and A.~Weber},
	\papertitle{Dynamics of the heat semigroup on symmetric spaces},
	\jourtit{Ergodic Theory Dynam. Systems}
    \journum{30}
     (2010), 457--468.

\bibitem{JiWeber2}
	\auth{L.~Ji and A.~Weber},
	\papertitle{$L^p$ spectrum and heat dynamics of locally symmetric spaces of higher rank},
	preprint Arxiv 2010.
     



\bibitem{Kawazoe}
	\auth{T.~Kawazoe and J.~Liu}, 
	\papertitle{Heat kernel and Hardy's theorem for Jacobi transform},
	\jourtit{Chinese Ann. Math. Ser. B}
	\journum{24} (2003),  359-Ð366. 

\bibitem{Koo}
	\auth{T.~H.~ Koornwinder},
	\papertitle{Jacobi functions and analysis on noncompact semisimple Lie groups},
  \booktit{Special functions: Group theoretical aspect and applications}. R. A. Askey et al. (eds.), Dordrecht-Boston: Reidel 1984, 1--85.
  
\bibitem{Liu}
	\auth{J.~Liu}, 
	\papertitle{Maximal functions associated with the Jacobi transform},
	\jourtit{Bull. London Math. Soc.}
	\journum{32} (2000),  582--588.  
	        

  
\bibitem{RS}
	\auth{S.~K.~Ray and R.~P.~Sarkar}, 
	\papertitle{Fourier and Radon transform on Harmonic $NA$ groups},
	\jourtit{Trans. Amer. Math. Soc.}
	\journum{8} (2009),  4269--4297.  
	        
    
\end{thebibliography}
\end{document}